\documentclass[a4paper]{article}

\usepackage{a4wide}

\usepackage[mathscr]{eucal}
\usepackage{amssymb}
\usepackage{latexsym}
\usepackage{amsthm}
\usepackage{amsmath}
\usepackage[dvips]{graphicx}
\usepackage{psfrag}

\newtheorem{theorem}{Theorem}[section]

\newtheorem{lemma}[theorem]{Lemma}

\theoremstyle{definition}
\newtheorem{definition}[theorem]{Definition}

\newtheorem{remark}[theorem]{Remark}

\newtheorem{example}[theorem]{Example}

\makeatletter
\@addtoreset{equation}{section}

\makeatother

\title{Inside $s$-inner product sets and Euclidean designs}
\author{Hiroshi Nozaki}

\begin{document}
\maketitle

\renewcommand{\thefootnote}{\fnsymbol{footnote}}
\footnote[0]{Supported by JSPS Research Fellow.}
\begin{abstract}
A finite set $X$ in the Euclidean space is called an $s$-inner product set if the set of
the usual inner products of any two distinct points in $X$ has size $s$. 
First, we give a special upper bound for the cardinality of an $s$-inner product set on concentric spheres. 
The upper bound coincides with the known lower bound for the size of a Euclidean $2s$-design. 
Secondly, we prove the non-existence of $2$- or $3$-inner product sets on two concentric spheres attaining the upper bound for any $d>1$.  
The efficient property needed to prove the upper bound for an $s$-inner product set gives the new concept, inside $s$-inner product sets. We characterize the most known tight Euclidean designs as inside $s$-inner product sets attaining the upper bound. 
\end{abstract}
\section{Introduction}
Delsarte-Goethals-Seidel \cite{Delsarte-Goethals-Seidel} gave a fundamental work 
exploring a new area within combinatorics for a finite subset of 
the unit sphere $S^{d-1}$. 
The two concepts of spherical $t$-designs and 
$s$-distance sets play important roles in their article. 

A finite subset $X$ of the Euclidean space $\mathbb{R}^d$ is called an $s$-distance set if the size of the set of the Euclidean distances between 
any two distinct points of $X$ is equal to $s$. 
We have a natural upper bound for the cardinality of an $s$-distance set in $S^{d-1}$, namely $|X| \leq \binom{d+s-1}{s}+\binom{d+s-2}{s-1}$.  
A basic problem for $s$-distance sets is to determine the maximum cardinality of $s$-distance sets for fixed $s$ and $d$. 

A spherical $t$-design is a set of points in $S^{d-1}$ satisfying that 
for any $d$-variable polynomial $f$ of degree at most $t$,
 the average of $f$ on the sphere is equal to the average of $f$ on the set. 
There is a natural lower bound for the size of a spherical $t$-design, and a design attaining this bound is said to be tight. 
 The bound for a $2e$-design is $|X|\geq \binom{d+e-1}{e}+\binom{d+e-2}{e-1}$.  
The primary purpose is to find the minimal design for fixed $t$ and $d$. 

One of main results in \cite{Delsarte-Goethals-Seidel} is that when
$|X|=\binom{d+s-1}{s}+\binom{d+s-2}{s-1}$, $X$ is an $s$-distance set in $S^{d-1}$ if and only if $X$ is a spherical 
$2s$-design. The classification of tight spherical $t$-designs is complete except for $t=4,5,7$ \cite{Bannai-Damerell1,Bannai-Damerell2,Bannai-Munemasa-Venkov}

This result is generalized in \cite{nozaki-shinohara} with a relationship between locally $s$-distance sets and weighted spherical $t$-designs. 
The two concepts have the same upper and lower bounds respectively. It follows that when
$|X|=\binom{d+s-1}{s}+\binom{d+s-2}{s-1}$, $X$ is a locally $s$-distance set if and only if $X$ is a weighted spherical 
$2s$-design. 
Actually, the weight function for a tight weighted spherical $t$-design is constant, and hence it becomes a 
spherical $t$-design. 
This implies that a locally $s$-distance set attaining the bound is an $s$-distance set. 

We expect that the theory is generalized to the Euclidean space $\mathbb{R}^d$. 
A key of the generalization is the existence of generalized upper bound and lower bound. 

The concept of Euclidean designs is introduced in \cite{Neumaier-Seidel}, and is known as a generalization of spherical designs to concentric spheres. 
A natural lower bound is well known, and a design attaining the bound is said to be tight. 
The classifications and the structures of tight designs are studied in many papers \cite{Bajnok2, Bajnok,Bannai-Bannai,Bannai-Bannai2,Bannai-Bannai-Hirao-Sawa,Bannai-Bannai-Shigezumi,Bannai-Bannai-Suprijanto,Etsuko,Etsuko2,Hirao-Sawa-Zhou}. 

We have two generalizations of spherical $s$-distance sets to $\mathbb{R}^d$. 
One is an $s$-distance set in $\mathbb{R}^d$, and the other is an $s$-inner product set in $\mathbb{R}^d$. 
There is an upper bound for $s$-distance sets on concentric spheres, and the upper bound coincides with 
the lower bound for Euclidean $2s$-designs \cite{Bannai-Bannai-Stanton, Bannai-Kawasaki-Nitamizu-Sato, Blokhuis}. 
Lison\v{e}k \cite{Liso} gave a $45$ point $2$-distance set in $\mathbb{R}^8$ attaining the upper bound. 
The example is a Euclidean $2$-design, and not a tight design. Other examples of $s$-distance sets 
attaining the bound have not been found so far. 

On the other hand, Deza and Frankl \cite{Deza-Frankl} proved the upper bound $|X| \leq \binom{d+s}{s}$ for a locally $s$-inner product set in $\mathbb{R}^d$. 
They state that ``As pointed out by the referee, Theorem 1.4 can be deduced also using 
the approach of Koornwinder'' in  \cite{Deza-Frankl}. 
However, a proof by this method has not been published. 
This problem was presented at the conference on Combinatorics, Geometry and Computer Science in 2007 \cite{Cameron}.

The first result of the present paper is to give a proof of the bound by the method of Koornwinder \cite{Koornwinder}. 
Moreover, the upper bound due to Deza--Frankl is improved as an upper bound for a locally $s$-inner product set on concentric spheres. 
The new upper bound coincides with 
the lower bound for Euclidean designs. 

In Section \ref{non}, we classify $2$- or $3$-inner product sets on two concentric spheres attaining the upper bound. Indeed, there dose not exist such an inner product set for any $d\geq 2$. 

The efficient property needed to prove the new bound for a locally $s$-inner product set gives the new concept, inside $s$-inner product sets. We can find a lot of examples of inside inner product sets attaining the upper bound. In particular, the most known tight Euclidean designs or ``modified'' tight Euclidean designs are inside inner product sets attaining the upper bound. This is a good characterization of several tight Euclidean designs with the view point of geometry.

\section{Preliminaries}
Let $\mathbb{R}^d$ denote the $d$-dimensional Euclidean space. 
For $x,y \in \mathbb{R}^d$, we denote their standard inner product by $(x,y)$.
Let $X$ be a finite subset of $\mathbb{R}^d$. We define 
$A(X):=\{(x,y) \mid  x,y \in X, x\ne y \}$.
For a fixed $x \in X$, we define 
$A(x):=\{(x,y) \mid  y \in X, x \ne y \}$
and $B(x):=\{(x,y) \mid  y \in X, x \ne y, ||x||\geq ||y|| \}$,
where $||x||:=\sqrt{(x,x)}$ is the norm of $x$. Let $|\ast|$ denote the cardinality. 
\begin{definition}
\begin{enumerate}
\item $X$ is called an $s$-inner product set if $|A(X)|=s$.  
\item $X$ is called a locally $s$-inner product set if $|A(x)| \leq s$ for each $x \in X$. 
\item $X$ is called an inside $s$-inner product set if $|B(x)| \leq s$ for each $x \in X$. 
\end{enumerate}
\end{definition}
Note that an $s$-inner product set is a locally $s$-inner product set, and a locally $s$-inner product
set is an inside $s$-inner product set. 

 Let $\mathcal{S}:=S_1\cup S_2 \cup \cdots \cup S_p$ be a union of $p$ concentric spheres, 
where $S_i$ is a sphere whose center is the origin and whose radius is $r_i$. 
We assume $0 \leq r_1 < r_2 < \cdots <r_p$. 
If $r_1$ is equal to zero, then $S_1$ is the origin and regarded as a special sphere. 
If $\mathcal{S}$ contains the origin, then $\varepsilon_{\mathcal{S}}:=1$, and if $\mathcal{S}$ does not contain the origin, then $\varepsilon_{\mathcal{S}}:=0$. 
For $X\subset \mathcal{S}$, we define $X_i:= X \cap S_i$. 
We say $X$ is supported by $\mathcal{S}$ if every $X_i$ is not empty.

Let ${\rm Hom}_l(\mathbb{R}^d)$ be the linear space of all real homogeneous polynomials of degree $l$, with $d$ variables. 
Define $\mathcal{P}_l(\mathbb{R}^d):= \oplus_{k=0}^l {\rm Hom}_k(\mathbb{R}^d)$, 
 $\mathcal{P}_l^{\ast}(\mathbb{R}^d):=\oplus_{i=0}^{\lfloor l/2 \rfloor} {\rm Hom}_{l-2i}(\mathbb{R}^d)$, 
 and ${\rm Harm}_l(\mathbb{R}^d):=\{f \in {\rm Hom}_l(\mathbb{R}^d) \mid \Delta f = 0 \}$, 
where $\Delta:=\sum_{i=1}^d {\partial^2}/{\partial x_i^2}$. 
An element of ${\rm Harm}_l(\mathbb{R}^d)$ is called a harmonic polynomial.    
Let $\mathcal{P}_l(\mathcal{S})$, ${\rm Hom}_l(\mathcal{S})$, ${\rm Harm}_l(\mathcal{S})$ and $\mathcal{P}_l^{\ast}(\mathcal{S})$ be the linear space of all functions which are the restrictions of
the corresponding polynomials to $\mathcal{S}$. 
For example, $\mathcal{P}_l(\mathcal{S}):=\{f|_{\mathcal{S}} \mid f \in \mathcal{P}_l(\mathbb{R}^d) \}$. 

The dimensions of these linear spaces are well known. 
Define $p'=p-\varepsilon_{\mathcal{S}}$.

\begin{theorem}[\cite{Bannai-Bannai-book,Delsarte-Seidel,Erdelyi}] 
\begin{enumerate}\setlength{\parskip}{0cm}
\item 
$
\dim\mathcal{P}_l(\mathcal{S})=\begin{cases}
\varepsilon_{\mathcal{S}} + \sum_{i=0}^{2p'-1} \binom{d+l-i-1}{d-1} \text{\qquad if $l \geq 2p'$}, \\
\dim \mathcal{P}_l(\mathbb{R}^d)=\binom{d+l}{l} \text{\qquad  if $l \leq 2p'-1$.} 
\end{cases}
$

\item 
$
\dim \mathcal{P}_l^{\ast}(\mathcal{S})=
\begin{cases}
\varepsilon_{\mathcal{S}}+\sum_{i=0}^{p'-1} \binom{d+l-2i-1}{d-1} \text{\qquad if $l$ is even and $l\geq 2p'$},\\
\sum_{i=0}^{p'-1} \binom{d+l-2i-1}{d-1} \text{\qquad if $l$ is odd and $l\geq 2p'$},\\
\dim \mathcal{P}_l^{\ast}(\mathbb{R}^d)=\sum_{i=0}^{\lfloor \frac{l}{2} \rfloor} \binom{d+l-2i-1}{d-1} \text{\qquad if $l \leq 2p'-1$}.
\end{cases}
$
\end{enumerate}
\end{theorem}

We consider the Haar measure $\sigma_i$ on each $S_i$. For $S_i \ne \{0\}$, we assume $|S_i|=\int_{S_i} d\sigma_i(x)$ where $|S_i|$ is the volume of $S_i$. If $S_i=\{0\}$, then we define $\frac{1}{|S_i|} \int_{S_i} f(x) d \sigma_i(x)=f(0) $. 
\begin{definition}[\cite{Neumaier-Seidel, Bannai-Bannai}]
Let $X$ be a finite set supported by $\mathcal{S} \subset \mathbb{R}^d$. Let $w(x): X \rightarrow \mathbb{R}_{>0}$ be a positive weight function. $(X,w)$ is called a Euclidean $t$-design if the following equality holds for any $f \in \mathcal{P}_t(\mathbb{R}^d)$:
$$
\sum_{i=1}^{p} \frac{w(X_i)}{|S_i|} \int_{S_i} f(x) d\sigma_i(x)=\sum_{x \in X} w(x)f(x).
$$
where $w(X_i):=\sum_{x \in X_i}w(x)$. 
\end{definition}
The largest value of $t$ for which $(X,w)$ is a Euclidean $t$-design is called the maximum strength of the design. If $p=1$, then a Euclidean $t$-design is called a weighted spherical $t$-design, and if $p=1$ and $w$ is a constant function, then a Euclidean $t$-design is called a spherical $t$-design. 

We have the Fisher type inequality for the cardinalities of Euclidean designs \cite{Bannai-Bannai-Hirao-Sawa, Delsarte-Seidel, Moller1, Moller2}. 
\begin{theorem} \label{ft for Euclidean}
\begin{enumerate}
\item Let $X$ be a Euclidean $2e$-design supported by $\mathcal{S}$. Then, 
\[
|X| \geq \dim \mathcal{P}_e(\mathcal{S}).
\]
\item Let $X$ be a Euclidean $(2e-1)$-design supported by $\mathcal{S}$. Then, 
\[
|X| \geq \begin{cases}
2 \dim \mathcal{P}_{e-1}^{\ast}(\mathcal{S})-1 \text{ \quad if $e$ is odd and $0\in X$}, \\
2 \dim\mathcal{P}_{e-1}^{\ast}(\mathcal{S}) \text{\quad \quad otherwise}.
\end{cases}
\]
\end{enumerate}
\end{theorem} 
A Euclidean $t$-design is said to be tight if it attains one of the lower bounds in Theorem \ref{ft for Euclidean}. 
If $X$ is a tight Euclidean $t$-design satisfying $0 \not \in X$, then we call $X\cup \{0\}$ an almost tight Euclidean $t$-design \cite{Bannai-Bannai-Hirao-Sawa}.    

\section{Upper bounds for an inner product set} 
In this section, we prove upper bounds for the size of an inside $s$-inner product set.
A finite $X \subset \mathbb{R}^d$ is said to be antipodal if for each $x \in X$, $-x$ is also an element of $X$. 
\begin{theorem} \label{main}
Let $\mathcal{S} \subset \mathbb{R}^d$ be a union of $p$ concentric spheres centered at the origin.
	\begin{enumerate}\setlength{\parskip}{0cm}
\item Let $X$ be an inside $s$-inner product set supported by $\mathcal{S}$. Then,
\[ |X| \leq \dim \mathcal{P}_s(\mathcal{S}). \]
\item Let $X$ be an antipodal inside $s$-inner product set supported by $\mathcal{S}$.  
Then, 
\[
|X|\leq 
\begin{cases}
2\, \dim \mathcal{P}_{s-1}^{\ast}(\mathcal{S}) + \varepsilon_{\mathcal{S}} \text{\qquad if $s$ is even},\\
2\, \dim \mathcal{P}_{s-1}^{\ast}(\mathcal{S})  \text{\qquad if $s$ is odd and $ 0 \not\in X $}.
\end{cases} 
\]
	\end{enumerate}
\end{theorem}

\begin{proof}
(1): 
Note that $-|| x||^2 \leq \alpha <||x||^2 $ for any $\alpha \in B(x)$. 
For each $x \in X$, we define the polynomial $f_x(\xi)$
in the variables $\xi=(\xi_1,\xi_2, \ldots ,\xi_d)$: 
\begin{equation} \label{f_x}
f_x(\xi)= \begin{cases}
		\prod_{\alpha \in B(x)} \frac{(x,\xi)-\alpha}{(x,x)-\alpha}& \text{ if $B(x) \ne \emptyset$},  \\
		1 \text{ (constant)}& \text{ otherwise}.	 
		\end{cases}
\end{equation}
Then, $f_x(\xi)$ is a polynomial of degree at most $s$. 
$B(x)$ is an empty set if and only if $x \in X_1$ and $|X_1|=1$. Hence, the number of $x$ such that $f_x(\xi)=1$ (constant) is at most $1$. 
It clearly follows that $f_x(x)=1$, and   
$f_x(y)=0$ for $x\ne y \in X$ and $||y||\leq ||x||$. 

We order the elements of $X$ as $X=\{x_1,x_2, \ldots, x_n \}$
in such a way that $||x_i||\leq ||x_{i+1}||$.
Let $M$ be the $n \times n$ matrix whose $(i,j)$-entry is $f_{x_i} 
(x_j)$. Then, $M$ is an upper triangular matrix whose diagonal entries are all one. 
This implies that 
$\{ f_{x_i} \}_{i=1,2,\ldots ,n}$ are linearly independent.
Therefore (1) follows.

(2):
There exists a subset $Y^{\prime}$ such that $X\setminus \{0\}
=Y^{\prime} \cup (-Y^{\prime})$, and $Y' \cap (-Y')$ is empty. We define  
$Y:=Y^{\prime} \cup \{0 \}$
or $Y^{\prime}$, according to $0\in X$ or not. Note that $|X| = 2|Y|- \varepsilon_{\mathcal{S}}$. For each $y \in Y$, we define
$
B^2(y):=\{\alpha^2 \mid \alpha \in B(y), \alpha \ne 0, \alpha \ne -(y,y)  \}.
$
Then, $|B^2(y)| \leq \lfloor (s-1)/2 \rfloor$. For any $\alpha^2\in B^2(y)$, we have $0 < \alpha^2 < (y,y)^2$. For each $y \in Y$, we define the polynomial $f_y(\xi)$: 
\[
f_y(\xi):=\begin{cases}
1 \text{ (constant) } \text{\qquad if $y =0$}, \\
\left( \frac{(y, \xi)}{(y,y)} \right) ^{(s-1)-2\lfloor (s-1)/2 \rfloor} \prod_{\alpha^2 \in B^2(y)} \frac{(y,\xi)^2-\alpha^2}{(y,y)^2-\alpha^2} \text{ \qquad otherwise}. 
\end{cases}
\]
If $0 \in Y$, and $s$ is even, then $1 \not\in \mathcal{P}_{s-1}^{\ast}(\mathbb{R}^d)$.
Therefore, $f_y(\xi) \in \mathcal{P}_{s-1}^{\ast}(\mathbb{R}^d) + \varepsilon_{\mathcal{S}}\,  {\rm Hom}_0(\mathbb{R}^d)$. Note that 
$f_y(y)=1$, 
and $f_y(z)=0$ for $y \ne z \in Y$ and $||z|| \leq ||y||$.  
By an argument similar to that in the proof of $(1)$, $\{f_{y}\}_{y \in Y}$ are linearly independent as elements of $\mathcal{P}_{s-1}^{\ast}(\mathcal{S}) + \varepsilon_{\mathcal{S}}\,  \rm{\rm{Hom}}_0(\mathcal{S})$. Hence, 
\begin{align*}
|Y| &\leq  \dim(\mathcal{P}_{s-1}^{\ast}(\mathcal{S}) + \varepsilon_{\mathcal{S}}\,  \rm{\rm{Hom}}_0(\mathcal{S}))\\
&=
\begin{cases}
\dim \mathcal{P}_{s-1}^{\ast}(\mathcal{S})+\varepsilon_{\mathcal{S}} \text{\qquad if $s$ is even}, \\ 
\dim \mathcal{P}_{s-1}^{\ast}(\mathcal{S}) \text{\qquad if $s$ is odd}. 
\end{cases}
\end{align*}
Since $|X|= 2|Y|-\varepsilon_{\mathcal{S}}$, (2) follows. 
\end{proof}
An inside $s$-inner product set $X$ is said to be tight, if $X$ attains one of the upper bounds in Theorem \ref{main}. 
\begin{remark}
Note that the upper bounds in Theorem \ref{main} coincide with the lower bounds in Theorem \ref{ft for Euclidean} for $s=e$, 
except when $\varepsilon_{\mathcal{S}} =1$ and $s$ is even. When $\varepsilon_{\mathcal{S}} =1$ and $s$ is even, the 
cardinality of a tight inside $s$-inner product set is equal to that of an almost tight Euclidean $(2s-1)$-design.   
\end{remark}
\begin{remark}
If $s \leq 2p'-1$, then the upper bound in Theorem~\ref{main} (1) 
coincides with Deza--Frankl's upper bound $|X|\leq \binom{d+s}{s}$ for a locally $s$-inner
product set. 
\end{remark}

\section{The non-existence of tight $2$- or $3$-inner product sets} \label{non}
In this section, we prove the non-existence of tight $2$- or $3$-inner product sets supported by a union of two concentric spheres. 
First, we show several results to prove the non-existence.  
\begin{theorem}\label{Rankin}
Let $X$ be a finite set in $\mathbb{R}^d$. 
\begin{enumerate}\setlength{\parskip}{0cm}
\item If $\alpha <0$ for all $\alpha \in A(X)$, then $|X| \leq d+1$. 
\item If $\alpha \leq 0$ for all $\alpha \in A(X)$, then $|X| \leq 2 d+1$.
\end{enumerate}
\end{theorem}
\begin{proof}
(1): Note that $0 \not\in X$. For any distinct $x,y \in X$, we have $(x/||x||,y/||y||)<0$, and $x/||x|| \ne y/||y||$. 
Therefore, $|X| \leq d+1$ by the Rankin bound \cite{Rankin}. 

(2): For any distinct non-zero elements $x,y \in X$, we have $(x/||x||,y/||y||) \leq 0$, and $x/||x|| \ne y/||y||$. 
Since $X$ may contain the origin, $|X| \leq 2d+1$ by the Rankin bound \cite{Rankin}.
\end{proof}
 
\begin{remark}
We can construct infinitely many examples attaining the bound in Theorem \ref{Rankin} (1). 
Examples attaining the bound in Theorem \ref{Rankin} (2) have the following forms:
\begin{equation}
X=\{0,  a_1 e_1,  a_2 e_2, \ldots ,a_d e_d, - a_{d+1} e_1,  - a_{d+2} e_2, \ldots ,- a_{2d} e_d  \}
\end{equation}
where $a_i$ are positive real numbers, and $\{e_i\}$ is an orthonormal basis of $\mathbb{R}^d$.    
\end{remark}
 
\begin{lemma}[\cite{Bannai-Munemasa-Venkov,Venkov}] \label{ven}
The following are equivalent: 
\begin{enumerate}
\item $X$ is a spherical $t$-design in $S^{d-1}$.  
\item For any $v \in \mathbb{R}^d$ and any $1 \leq l \leq t$, 
\[
\sum_{x \in X} (v,x)^l= 
\begin{cases}
0 \qquad \text{ if $l$ is odd}, \\
\frac{(l-1)!!(d-2)!!}{(d+l-2)!!}|X|(v,v)^{\frac{l}{2}} \qquad \text{ if $l$ is even}.
\end{cases}
\]
\end{enumerate}
\end{lemma}
 
\begin{lemma} \label{key lemma2}
Let $X=\{x_1,x_2, \ldots , x_n\}$ be an inside $s$-inner product set supported by $\mathcal{S} \subset \mathbb{R}^d$. Let $\{\varphi_i \}_{1 \leq i \leq v}$ be a basis of $\mathcal{P}_s(\mathcal{S})$, and $v$ be the dimension of $\mathcal{P}_s(\mathcal{S})$. Let $M$ be the $n \times v$ matrix whose $(i,j)$-entry is $\varphi_j(x_i)$. 
Then, the rank of $M$ is $n$. In particular, if $n=v$, then $M$ is a nonsingular matrix. 
\end{lemma}
\begin{proof}
Let $f_{x_k}(\xi)$ be the $d$-variable polynomial in (\ref{f_x}). 
Since $f_{x_k}(\xi)$ is of degree at most $s$, we can write $f_{x_k}(\xi)=\sum_{i=1}^v a_i^{(x_k)} \varphi_i(\xi)$, where $a_i^{(x_k)}$ are real numbers. Let $N$ be the $n \times v$ matrix $(a_j^{(x_i)})_{i,j}$. Then, $NM^T$ is an upper triangular matrix whose diagonal entries are all $1$, and hence it is of rank $n$. This implies that $M$ is of rank $n$.  
\end{proof}

Define $F_X(t)=\sum_{\alpha \in A(X)} (t- \alpha)$ for $X \subset \mathbb{R}^d$. 

\begin{lemma} \label{f_i}
Let $X$ be an $s$-inner product set in $\mathbb{R}^d$.  
We have the expression $F_X(t)= \sum_{i=0}^s f_i t^i$, where $f_i$ are real numbers.  
If $|X| >2d+1$, then there exists $i$ such that $f_i < 0$. 
\end{lemma}
\begin{proof}
By Theorem \ref{Rankin}, there exists $\alpha \in A(X)$ such that $\alpha > 0$. 
If $f_i \geq 0$ for all $0 \leq i \leq s$, then $\sum_{i=0}^s f_i t^i$ is monotonically increasing for $t>0$. 
Since $f_0 \geq 0$, this contradicts $\alpha>0$.  
\end{proof}

\begin{lemma} \label{s=2}
Let $X$ be a tight $s$-inner product set supported by a union of two concentric spheres, which dose not contain the origin, 
and is not antipodal, for $s=2,3$. We have the form $F_X(t)=\sum_{i=0}^s f_i t^i$. 
Then, 
\[
|X_1|=\sum_{i: f_i < 0} h_i, \qquad |X_2|=\sum_{i: f_i > 0} h_i, 
\]
where $h_i= \dim{\rm Hom}_i(R^d)$. 
\end{lemma}
\begin{proof}
Define $\varphi_{\lambda}(x):=\sqrt{\binom{|\lambda|}{\lambda_1,\lambda_2, \ldots, \lambda_d}}x_1^{\lambda_1}x_2^{\lambda_2}
 \cdots 
x_d^{\lambda_d}$ for
 $x=(x_1,x_2,\ldots, x_d) \in \mathbb{R}^d$,
$\lambda=(\lambda_1,\lambda_2,\ldots,\lambda_d) \in \mathbb{Z}_{\geq 0}^d$, and $|\lambda |=\sum_{i=1}^d \lambda_i$. 
Let $H_l$ be the matrix indexed by $X$ and $\{ \varphi_{\lambda} \mid |\lambda|=l \}$, whose $(x,\lambda)$-entry is $\varphi_{\lambda}(x)$. 
Let $\mathcal{H}=[H_0,H_1,\ldots,H_s]$ and
 $M=f_0 I_1 \oplus f_1 I_{h_1} \oplus \cdots \oplus f_s I_{h_s}$ (a direct sum), where $I_{k}$ is the identity matrix of size $k$.  
Then, we may write 
\begin{equation} \label{eq_key}
\mathcal{H}M\mathcal{H^T} = \left[\begin{array}{cc}  
F_X(r_1^2) I_{|X_1|} & 0\\
 0                              &      F_X(r_2^2) I_{|X_2|}\\
\end{array}
\right].
\end{equation}
Since $|X|=\sum_{i=0}^s h_i$ for $s=2,3$, and Lemma \ref{key lemma2}, $\mathcal{H}$ is a nonsingular matrix. 
The numbers of positive, negative and zero eigenvalues of $M$ are equal to those of the right hand side of (\ref{eq_key}) respectively. Note that $F_X(r_2^2)>0$. By Theorem \ref{Rankin} and Lemma \ref{f_i}, some $f_i$ are negative. This implies that $F_X(r_1^2)<0$. Therefore, the lemma follows.  
\end{proof}
\begin{remark}
Since $F_X(r_1^2)<0$, odd number of inner products are greater than $r_1^2$.  
Therefore,  for $s=2$ (resp.\ $s=3$), $X_1$ is a $1$-inner product set (resp.\ $2$-inner product set) or $|X_1|=1$. 
There are a few possible pairs $|X_1|$ and $|X_2|$ by Lemma \ref{s=2}. 
The signs of $f_i$ give some conditions of inner products. 
\end{remark}
The following are the main theorems in this section.  
\begin{theorem} 
There does not exist a tight $2$-inner product set, that is supported by a union of two concentric spheres and is not antipodal, for any $d\geq 2$. 
\end{theorem}
\begin{proof}
Let $A(X)=\{\alpha, \beta\}$ where $\alpha > \beta$. 
If $X$ contains the origin, then $X_2$ is a tight spherical $4$-design. 
Then $A(X_2)$ does not contain zero, a contradiction.  

Assume $|X_1|=1$ and $|X_2|=d+\binom{d+1}{2}$. Then $X_2$ is a tight spherical $4$-design. 
Applying Lemma \ref{ven} to $X_2$ and $v \in X_1$, we obtain a contradiction. 

If $|X_1|=d$ and $|X_2|=1 + \binom{d+1}{2}$, then $\alpha + \beta > 0$. We have $|X_2| \leq \binom{d+1}{2}$ \cite{Musin, nozaki-shinohara}, a contradiction. 

Assume $|X_1|=d+1$ and $|X_2|=\binom{d+1}{2}$. 
Since $X_1$ is a $1$-inner product set and $|X_1|=d+1$, $X_1$ is a tight spherical $2$-design. Without loss of generality, we may assume $r_1=1$ and hence $\beta=-1/d$. Applying Lemma \ref{ven} to $X_1$ and $v \in X_2$, 
we obtain a contradiction. 
\end{proof}

\begin{theorem} 
There does not exist a tight $3$-inner product set, that is supported by a union of two concentric spheres and is not antipodal, for any $d\geq 2$. 
\end{theorem}

\begin{proof}
Let $A(X)=\{\alpha, \beta, \gamma \}$ where $\alpha > \beta> \gamma$. If $X$ contains the origin, then $X_2$ is a tight spherical $6$-design. 
A tight spherical $6$-design does not exist, except for the regular heptagon on $S^1$. 
Thus, no tight $3$-inner product set exists on $\{0\}\cup S_2$. 

We deal with the non-trivial cases. 

Assume $|X_1|=d+\binom{d+1}{2}$ and $|X_2|=1+ \binom{d+2}{3}$. 
Since $X_1$ is a $2$-distance set and $|X_1|=d+\binom{d+1}{2}$, $X_1$ is a tight spherical $4$-design. 
We may assume $r_1=1$, and $A(X_1)=\{\frac{-1 \pm \sqrt{d+3}}{d+2} \}$.
Applying Lemma \ref{ven} to $X_1$ and $v \in X_2$, we obtain a contradiction. 
 
Assume $|X_1|=1+\binom{d+1}{2}$ and $|X_2|=d+\binom{d+2}{3}$. 
Then, $\alpha>0$, $\beta<0$, $\gamma<0$, $\alpha+\beta+\gamma > 0$, and $\alpha \beta+\beta \gamma+\gamma \alpha > 0$.
We have $\alpha \beta+ \beta \gamma + \gamma \alpha <  - (\beta+ \gamma)^2 + \beta \gamma= -\beta^2 - \beta \gamma -\gamma^2 < 0$, a contradiction. 

Assume $|X_1|=1+d$ and $|X_2|=\binom{d+1}{2}+\binom{d+2}{3}$. Then, $X_2$ is a tight spherical $6$-design, a contradiction. 

Assume $|X_1|=\binom{d+1}{2}$ and $|X_2|=1+d+\binom{d+2}{3}$. Then, $\alpha+\beta+\gamma>0$. We have the Gegenbauer expansion $(t-\alpha)(t-\beta)(t-\gamma)=\sum_{k=0}^3 a_k G_k^{(d)}(t)$ where $a_2=-2(\alpha+\beta+\gamma)/d(d+2)$. 
Therefore, $|X_2| \leq \binom{d+2}{3}+1$ \cite{nozaki-shinohara}, a contradiction.

\end{proof}
\section{Tight inside $s$-inner product sets} \label{ex_sec}
In this section, we introduce several examples of tight inside inner product sets, and their maximum achievable strengths as Euclidean designs whose weight functions are constant on each $X_i$. 
 
A tight $1$-inner product set with a negative inner product is identified with a tight Euclidean $2$-design \cite{Bannai-Bannai-Suprijanto}. 

A tight locally $2$-inner product set is identified with a tight or almost tight Euclidean $3$-design \cite{Etsuko2}.   

A tight inner product set on a sphere is a tight spherical design \cite{Delsarte-Goethals-Seidel}. 
The union of a tight spherical $(4m-1)$-design \cite{Delsarte-Goethals-Seidel} and the origin is a tight inside $2m$-inner product set.   
 
We have several non-trivial examples of tight inside $s$-inner product sets (Tables $1,2,3,4$).  
A lot of known tight Euclidean designs are tight inside $s$-inner product sets. 
Modifying the radii of the spheres supporting tight Euclidean design to reduce the number of the inner products, we may obtain a tight inside inner product set (we write M in the tables).

\begin{table}
{\small T: tight design,  Al: almost tight design, M: modified tight design\\
L: locally inner product set, A: antipodal set, $B(x)=B(X_i)$ for each $x \in X_i$} \\
{ \small 
\begin{tabular}{|c|c|c|c|c|c|c|c|c|c|c|} 
 \hline
$s$ &$d$&$t$&$|X_1|$&$|X_2|$&$B(X_1)$                     &$B(X_2)$& $r_2$   & $w_2$ & remark & ref.      \\  \hline

$2$   &$2$  & $4$ &$3$      & $3$     &$-\frac{1}{2}$               &$-2, 1$ &$2$ &     $\frac{1}{8} $ &T & \cite{Bannai-Bannai, Etsuko}   \\
    &$4$  & $4$ &$10$     &$5 $     &$-\frac{2}{3}, \frac{1}{6}$&$-\frac{3}{2}, 1$&$\sqrt{6}$& $\frac{1}{27}$&T & \cite{Etsuko} \\
    &$5$  &$ 4$ &$6$      &$15$     &$-\frac{1}{5}$               &$-\frac{4}{5}, \frac{2}{5}$&$\sqrt{\frac{8}{5}}$ &$\frac{1}{2}$ &T& \cite{Etsuko}\\
    &$4$  & $4$ &$6$      &$9$      &$-\frac{1}{2}, 0$&$-1, \frac{1}{2}$ &$\sqrt{2}$ & $\frac{1}{3}$&T&\cite{Etsuko}\\
    &$22$ &$3$  & $253$   & $23$    &$-\frac{13}{56}, \frac{5}{28}$ &$-\frac{16}{7}, 1$ &$4 \sqrt{\frac{22}{7}}$ & $\frac{3}{512}$ & M & \cite{Etsuko} \\ 
    &$d\geq 3$  & $2$ &$\frac{d+1}{2}$ & $d+1$& $-\frac{2}{d-1}, \frac{d-3}{2(d-1)}$ & $-\frac{2}{d-1}, 1$& $\sqrt{\frac{2d}{d-1}}$ 
    & any & non & Ex.\ \ref{ex1} \\
    &$d\geq 5$ & $3$& $\frac{d+1}{2}$ & $d+1$& $-\frac{2}{d-1}, \frac{d-3}{2(d-1)}$& $-\frac{d-1}{2}, 1$& $\sqrt{\frac{d(d-1)}{2}}$& $\frac{d-3}{(d-1)^3}$&$d=6$: M & Ex.\ \ref{ex2}\\
 \hline
3
&$3$&$5$&$6$&$8$&$-1, 0$&$-3, \pm 1$&$\sqrt{3}$ & $\frac{1}{8}$ & A, L, T & \cite{Bajnok2, Bajnok, Etsuko2}\\
&$5$&$5$&$12$&$20$&$-1, \pm \frac{1}{5}$ &$-\frac{9}{5}, \pm \frac{3}{5} $ &$\frac{3}{\sqrt{5}}$  &$\frac{1}{3}$& A, T& \cite{Etsuko2}\\
&$5$&$5$&$20$&$12$&$-1, \pm \frac{1}{3}$ &$-5, \pm 1 $ &$\sqrt{5}$ &$\frac{1}{27}$ &A, T&\cite{Etsuko2} \\
&$6$&$5$&$12$&$32$&$-1, 0$ &$-\frac{3}{2}, \pm\frac{1}{2}$ &$\sqrt{\frac{3}{2}}$ &$\frac{1}{2}$ & A, T
& \cite{Etsuko2} \\
&$22$&$5$&$2025$&$275$& $-\frac{4}{11}, \-\frac{1}{44}, \frac{7}{22}$& $-\frac{9}{11}, -\frac{3}{22}, \frac{6}{11}$& $\frac{6}{\sqrt{11}}$ &$\frac{3}{32}$ & M & \cite{Bannai-Bannai-Shigezumi} \\
\hline
$4$
& $2$ & $7$ &   $6$   &   $6$   &$-1, \pm \frac{1}{2}$&$-3,0,\pm \frac{3}{2}$  & $\sqrt{3}$ & $\frac{1}{27}$ &  A, T & \cite{Bajnok2} \\ 
&$4$&$7$&$24$&$24$&$-1, 0, \pm \frac{1}{2}$&$-2, 0, \pm 1$ &$\sqrt{2}$ &$\frac{1}{8}$ & A, T &\cite{Bannai-Bannai2}\\
&$7$&$7$&$56$&$126$&$-1,  \pm \frac{1}{3}$&$-\frac{4}{3}, 0, \pm \frac{2}{3}$ &$\frac{2}{\sqrt{3}} $  &$\frac{1}{2}$ & A, T& \cite{Bannai-Bannai2}\\
&$7$&$7$&$126$&$56$&$-1, 0, \pm \frac{1}{2}$&$-3, 0, \pm 1$ &$\sqrt{3} $ &$\frac{1}{32}$ & A, T &\cite{Bannai-Bannai2} \\
\hline
\end{tabular}}\\ 
\caption{$p=2$}
\end{table}

\begin{table} 
{\small
\begin{tabular}{|c|c|c|c|c|c|c|c|c|c|c|c|c|c|c|}  \hline
$s$    & $d $    &$t$ &$|X_1|$&$|X_2|$&$|X_3|$&$B(X_1)$&$B(X_2)$ & $B(X_3)$ &$r_2$&$r_3$      &$w_2$&$w_3$ &rem. & ref. \\  \hline
$4$&$3$&$7$&$12$&$6$&$8$&$-1, 0, \pm \frac{1}{2}$&$-2, 0, \pm 1$ &$-6, 0, \pm 2$ &$\sqrt{2}$&$\sqrt{6}$ & $\frac{5}{32}$& $\frac{1}{256}$& A, T&\cite{Bajnok}\\ \hline
\end{tabular}
}
\caption{$p=3$ ($p=4, 0 \in X$)}
\end{table}

\begin{table} 
{\small
\begin{tabular}{|c|c|c|c|c|c|c|c|c|c|c|c|}  \hline
$s$&$d$&$t$&$|X_1|$&$|X_2|$&$B(X_1)$                       &$B(X_2)$         &$r_2$    & $w_2$  &rem.  &  ref.  \\  \hline
$4$  & $2$ & $7$ &   $6$   &   $6$   &$-1,0, \pm \frac{1}{2}$&$-3,0,\pm \frac{3}{2}$  & $\sqrt{3}$ & $\frac{1}{27}$ &  A, Al & \cite{Bajnok2, Bannai-Bannai-Hirao-Sawa} \\ 
&$4$&$7$&$24$&$24$&$-1, 0, \pm \frac{1}{2}$&$-2, 0, \pm 1$ &$\sqrt{2}$ &$\frac{1}{8}$ & A, Al &\cite{Bannai-Bannai2,Bannai-Bannai-Hirao-Sawa}\\
&$7$&$7$&$56$&$126$&$-1, 0, \pm \frac{1}{3}$&$-\frac{4}{3}, 0, \pm \frac{2}{3}$ &$\frac{2}{\sqrt{3}} $  &$\frac{1}{2}$ & A, Al& \cite{Bannai-Bannai2,Bannai-Bannai-Hirao-Sawa}\\
&$7$&$7$&$126$&$56$&$-1, 0, \pm \frac{1}{2}$&$-3, 0, \pm 1$ &$\sqrt{3} $ &$\frac{1}{32}$ & A, Al &\cite{Bannai-Bannai2,Bannai-Bannai-Hirao-Sawa} \\ \hline
\end{tabular}}
\caption{$p=3$, $0 \in X$}
\end{table}

\begin{table}
{\small
\begin{tabular}{|c|c|c|c|c|c|c|c|c|c|c|c|c|} \hline
$s$&$d$&$t$&$|X_1|$&$|X_2|$&$|X_3|$&$B(X_1)$ & $B(X_2)$ & $B(X_3)$&$r_2$ &$r_3$ & rem. &ref. \\  \hline
$2$  & $2$ & $0$ & $2$     & $1$     & $2$     & $0$       & $0, 1$     & $0,2$     &$\sqrt{2}$ & $2$ &non&Ex.\ \ref{ex3} \\ \hline
\end{tabular}}
\caption{$p=4$, $0 \in X$}
\end{table}

\begin{example}[Table 1]\label{ex1}
Let $X_2$ be the $d$-dimensional regular simplex in $\mathbb{R}^d$ and  
$
X_1:=\{ (x+y)/2 \mid x , y \in X_2, x \ne y \}.
$
\end{example}

\begin{example}[Table 1]\label{ex2}
Let $X_2$ be the $d$-dimensional regular simplex in $\mathbb{R}^d$ and  
$
X_1:=\{ -(x+y)/(d-1) \mid x, y \in X_2, x \ne y \}.
$
\end{example}

\begin{example}[Table 4]\label{ex3}
Let $X_1:=\{ (0,0) \}$, $X_2:=\{ (1,0),(0,1) \}$, $X_3:=\{ (1,1) \}$ and $X_4:=\{(2,0),(0,2) \}$. 
\end{example} 
\begin{remark}
We expect a closely relationship between tight inside inner product sets and tight Euclidean designs. 
We can observe that a lot of examples of tight Euclidean designs have the structure of tight inside inner product sets. 
However, there are tight Euclidean designs which are not related with tight inside inner product sets. 
Tight or almost tight Euclidean designs in $\mathbb{R}^2$ \cite{Bajnok, Bannai-Bannai-Hirao-Sawa} except the examples in the tables, 
and tight Euclidean $t$-designs with $(t,d,|X_1|,|X_2|)=(4,22,33,243)$ \cite{Etsuko}, $(t,d,|X|,p)=(4,4,22,3)$ \cite{Hirao-Sawa-Zhou} cannot become tight inside $s$-inner product sets even if we modify the radii. 
Conversely, some tight inside inner product sets are not even $1$-design. 
Tight $1$-inner product sets with a non-negative inner product and Example $\ref{ex3}$ are not even $1$-designs. 
To construct of a generalized theory of the sphere due to Delsarte-Goethals-Seidel \cite{Delsarte-Goethals-Seidel}, 
we have to give an additional condition for Euclidean designs or inside inner product sets.  
\end{remark}

 \textbf{Acknowledgements.} The author thanks Professor Eiichi 
Bannai and Professor Akihiro Munemasa for providing useful
information and comments. 
The author is grateful to the referees for a lot of insightful suggestions.

{\it Hiroshi Nozaki}\\
	Graduate School of Information Sciences, \\
	Tohoku University \\
	Aramaki-Aza-Aoba 6-3-09, \\
	Aoba-ku, \\
	Sendai 980-8579, \\
	Japan\\ 
	nozaki@ims.is.tohoku.ac.jp\\
\end{document}